\documentclass[11pt]{article}

\usepackage[latin1]{inputenc}
\usepackage{t1enc}
\usepackage{amsmath, amssymb, amsfonts, amsthm, amsopn,epsfig}
\usepackage[none]{hyphenat}
\usepackage{tikz}
\usepackage{float}
\usepackage{graphicx}
\usepackage{caption}
\usepackage[toc,page]{appendix}
\usepackage{epstopdf}
\usepackage{etoolbox}
\usepackage[colorlinks=true]{hyperref}
\usepackage{dsfont}
\hypersetup{
	colorlinks=true,
	linkcolor=blue,
	filecolor=magenta,
	urlcolor=cyan,
	citecolor=blue
}
\usepackage{cleveref}
\usepackage[top=20mm, bottom=20mm, left=23mm, right=23mm]{geometry}
\usepackage{comment}

\theoremstyle{plain}
\newtheorem{theorem}{Theorem}[section]

\newtheorem{claim}[theorem]{Claim}
\newtheorem{lemma}[theorem]{Lemma}

\theoremstyle{definition}

\newcommand{\Pb}{\mathbb{P}}

\newcommand{\hide}[1]{}

\title{
\vspace{-0.7cm}
Long induced paths in $K_{s, s}$-free graphs}
\author{
Zach Hunter \thanks{Department of Mathematics, ETH Z\"urich, Switzerland. Email: {\tt \{zach.hunter, aleksa.milojevic, benjamin.sudakov\}@math.ethz.ch}. Research supported in part by SNSF grant 200021-228014.}
\and Aleksa Milojevi\'c\footnotemark[1] \and Benny Sudakov \footnotemark[1]
\and 
Istv\'an Tomon\thanks{Ume\r{a} University, \emph{e-mail}: \textbf{istvantomon@gmail.com}, Research supported in part by the Swedish Research Council grant VR 2023-03375.}}
\date{}

\begin{document}

\maketitle

\begin{abstract}
More than 40 years ago, Galvin, Rival and Sands showed that every $K_{s, s}$-free graph containing an $n$-vertex path must contain an induced path of length $f(n)$, where $f(n)\to \infty$ as $n\to \infty$. Recently, it was shown by Duron, Esperet and Raymond that one can take $f(n)=(\log \log n)^{1/5-o(1)}$. In this note, we give a short self-contained proof that a $K_{s, s}$-free graph with an $n$-vertex path contains an induced path of length at least $(\log \log n)^{1-o(1)}$. Combined with the recent remarkable example of Cou\"etoux, Defrain, and Raymond, which provides an upper bound of $O((\log \log n)^{1+o(1)})$, this essentially resolves this old problem.
\end{abstract}

\section{Introduction}

In 1982, Galvin, Rival and Sands \cite{GRS} showed that every infinite graph $G$ with a Hamilton path has either arbitrarily long induced paths or it contains the infinite half-graph as a subgraph. Motivated by this result, they asked the following finitary problem: given a graph $G$ that does not contain the complete bipartite graph $K_{s,s}$, and which contains a path on $n$ vertices, how long of an induced path must $G$ contain?
They gave a proof showing that every such $G$ contains an induced path of length at least $\Omega\big((\log\log\log n)^{1/3}\big)$ (here and later, the $\Omega(.)$, $O(.)$, and $o(.)$ notations assume that $s$ is a fixed constant). 

This statement was rediscovered some 30 years later, with worse quantitative bounds, by Atminas, Lozin and Razgon \cite{ALR}, in the context of parameterized complexity of the biclique problem. Very recently, Duron, Esperet and Raymond \cite{DER} improved upon the bounds of \cite{GRS}, showing that a $K_{s, s}$-free graph $G$ with an $n$-vertex path must contain an induced path of length at least $(\log\log n)^{1/5-o(1)}$. Recall, a graph is called $K_{s, s}$-free if it does not contain $K_{s, s}$ as a (not necessarily induced) subgraph. The purpose of our paper is to give a very short proof of the following improved lower bound. 

\begin{theorem}\label{thm:main}
Let $n\geq s\geq 2$ be positive integers and let $G$ be a $K_{s, s}$-free graph containing a path on $n$ vertices. Then $G$ contains an induced path of length at least $\Omega \big(\frac{\log \log n}{\log\log\log n}\big)$.
\end{theorem}

\noindent
This theorem is nearly optimal, and, together with the remarkable example of Cou\"etoux, Defrain and Raymond, it essentially resolves the problem of Galvin, Rival and Sands. More precisely, Cou\"etoux, Defrain and Raymond \cite{CDR} (building on the earlier work of Defrain and Raymond \cite{DR}) show the existence of $2$-degenerate $n$-vertex graphs with a Hamilton path in which all induced paths have length $O(\log\log n\cdot \log\log\log n)$. Note that $2$-degenerate graphs are also $K_{3,3}$-free. 

Determining the length of longest induced paths in hereditary classes of graphs has been the subject of many recent papers, see \cite{DER} for an excellent overview of the results related to this problem. A result of Ne\v{s}et\v{r}il and Ossona de Mendez \cite{NOdM} addresses the class of $d$-degenerate graphs, showing that a $d$-degenerate graph must contain an induced path of length $\Omega(\log \log n/\log d)$. As mentioned above, this bound cannot be improved beyond $\log\log n\cdot \log\log \log n$ already for $d=2$ by a construction of Cou\"etoux, Defrain and Raymond \cite{CDR}.

More generally, the study of long induced paths in graphs with structural properties fits well into the broader line of research about induced subgraphs of $K_{s, s}$-free or $K_s$-free graphs. The assumption that the host graph $G$ is $K_{s, s}$-free or $K_s$-free is a natural one, since it forbids the host graph from being a complete graph, which contains no nontrivial induced subgraphs. An example of this type of problem is the old question of Erd\H{o}s, Saks and S\'os \cite{ESS}, which asks to determine the maximum size of an induced tree in a connected $K_s$-free graph. This problem exhibits an interesting transition between $s=3$ and $s\geq 4$, since the largest induced tree in a $K_3$-free graph is of size $\Omega(\sqrt{n})$, while for $s\geq 4$ there are $K_s$-free graphs without induced trees on more than $O(\log n)$ vertices, see \cite{FLS}. 

The systematic study of induced subgraphs of $K_{s, s}$-free graphs was initiated by the authors of this note in the recent paper \cite{HMST}, posing the following general Tur\'an-type question. Given a fixed graph $H$, what is the maximum number of edges in a $K_{s, s}$-free graph $G$ on $n$ vertices which does not contain $H$ as an induced subgraph? If $s$ and $H$ are fixed, it is shown for several interesting classes of graphs, such as trees or cycles, that this maximum number of edges cannot exceed the usual extremal number ${\rm ex}(n, H)$ by more than a constant factor depending on $s$ and $H$, and it was conjectured that this phenomenon holds for all bipartite graphs $H$.

Finally, let us mention that the study of $K_{s, s}$-free graphs without certain induced subgraphs also comes up naturally in geometry. Namely, incidence or intersection graphs of various geometric objects often avoid certain induced subgraphs for geometric reasons, and thus the maximum number of edges needed to find a $K_{s, s}$ can be studied through the methods of graph theory. For more detailed discussions of this subject, see the excellent survey of Smorodinsky \cite{S}.

\section{Long induced paths}

In this section, we prove Theorem \ref{thm:main} after a few preliminary lemmas. Given a graph $G$ and $v\in V(G)$, $N(v)$ denotes the neighbourhood of $v$, $\deg(v)=|N(v)|$ denotes the degree of $v$ and $e(G)=|E(G)|$ denotes the number of edges of $G$.

\begin{lemma}\label{lemma:sparse codegree graph}
Let $d, s\geq 2$ be integers and $t\geq 1$ be a real number satisfying $t\leq d^{1/s}/s$. If $G$ is a $K_{s, s}$-free graph of minimum degree at least $d$ and $v$ is a vertex of $G$, then less than $st$ vertices $u\in V(G)$ satisfy $|N(u)\cap N(v)|\geq \frac{4}{t} |N(v)|$.
\end{lemma}
\begin{proof}
Suppose for the sake of contradiction that there exists a set $U$ of size $|U|=st$ such that for every $u\in U$, we have $|N(u)\cap N(v)|\geq \frac{4}{t} |N(v)|$. Consider the bipartite subgraph $H$ of $G$ induced between the sets $U$ and $N(v)\backslash U$.

Each vertex of $U$ has at least $\frac{4}{t}|N(v)|-|U|\geq \frac{2}{t}|N(v)|$ neighbours in $N(v)\backslash U$. Here, we used that $|U|=st\leq \frac{2}{t}d\leq \frac{2}{t}|N(v)|$, which follows from the the assumptions $t\leq d^{1/s}/s$ and $s\geq 2$. Hence, the number of edges in $H$ is at least $|U|\cdot \frac{2}{t}|N(v)|$. On the other hand, the K\H{o}v\'ari-S\'os-Tur\'an theorem implies that in a $K_{s, s}$-free bipartite graph with parts of size $m=|U|$ and $n=|N(v)\backslash U|$, one can have at most $(s-1)^{1/s} mn^{1-1/s}+(s-1)n$ edges (see e.g. Theorem 2.2 in Chapter VI of \cite{B}). In our case, 
\begin{equation}\label{eqn:KST}
    \frac{2}{t} |U| |N(v)|\leq e(U, N(v)\backslash U)< s |U||N(v)|^{1-1/s} + s|N(v)|.
\end{equation}
But  $\frac{1}{t}|U||N(v)|\geq s |U| |N(v)|/d^{1/s}\geq s |U||N(v)|^{1-1/s}$, since $|N(v)|\geq d$ and $\frac{1}{t} |U| |N(v)|=s|N(v)|$. This contradicts (\ref{eqn:KST}), finishing the proof.
\end{proof}

\begin{lemma}\label{lemma:min degree}
Let $G$ be a $K_{s, s}$-free graph of minimum degree at least $d$. Then $G$ contains an induced path of length at least $d^{1/{2s}}/2s$.
\end{lemma}
\begin{proof}
Fix $k=\lceil \frac{1}{2s}d^{1/{2s}}\rceil$ and assume $k\geq 4$ since the statement is trivial otherwise. We pick a random walk $v_1, \dots, v_k$ according to the stationary distribution. Namely, the starting vertex $v_1$ is chosen at random from $V(G)$ such that $\Pb[v_1=w]=\frac{\deg w}{2e(G)}$ for all $w\in V(G)$, and each subsequent vertex $v_{i+1}$ is chosen uniformly at random from the neighbours of $v_i$. Note that this choice ensures that $(v_i,v_j)$ has the same distribution as $(v_j,v_i)$ for any $1\leq i<j\leq k$ (this follows since the random variables $v_1, \dots, v_k$ are identically distributed and the walk $(v_1, \dots, v_k)$ has the same distribution as its reverse $(v_k, \dots, v_1)$). We show that the walk $v_1, \dots, v_k$ is an induced path with positive probability. 

Define an auxiliary directed graph $H$ on the vertex set $V(G)$, where $v \to u$ if $|N(u)\cap N(v)|\geq \frac{4s}{d^{1/s}} |N(v)|$. Lemma~\ref{lemma:sparse codegree graph} applied with $t=d^{1/s}/s$ shows that the outdegree of a vertex $v$ in $H$ is at most $d^{1/s}$. Next, we estimate the probability that there is no edge $v_j\rightarrow v_i$ for any $1\leq i<j\leq k$ in $H$, and we denote this event by $E$. By the union bound, we have 
\[\Pb\big[\,\overline{E}\,\big]=\Pb[v_j\to v_i\text{ for some } 1\leq i<j\leq k]\leq \sum_{1\leq i<j\leq k} \Pb[v_j\to v_i].\]
Recall that the pair $(v_i,v_j)$ has the same distribution as $(v_j,v_i)$, so $\Pb[v_j\to v_i]=\Pb[v_i\to v_j]$. Since $v_j$ is chosen uniformly from $N_G(v_{j-1})$, which is a set of size at least $d$, and the outdegree of $v_i$ in $H$ is at most $d^{1/s}$, we find that $\Pb[v_i\to v_j]\leq \frac{d^{1/s}}{d}=d^{1/s-1}$. Thus,  $\Pb[\overline{E}]\leq \binom{k}{2} d^{1/s-1}\leq k^2d^{1/s-1}$.

The next step is to estimate the probability that $v_1, \dots, v_k$ is an induced path. The walk $v_1, \dots, v_k$ is an induced path if and only if there are no edges between $v_i$ and $v_j$ for $i<j-1$. Indeed, since $k\geq 4$, the latter condition implies that the vertices $v_1, \dots, v_k$ are all distinct. Fixing some $1\leq i<j-1\leq k-1$, our main observation is that $\Pb[v_iv_j\in E(G) | v_{j-1}\not\to v_{i}]\leq 4sd^{-1/s}.$ Indeed, if $v_{j-1}\not \to v_{i}$ in $H$, then at most $\frac{4s}{d^{1/s}}|N(v_{j-1})|$ neighbours $v_j\in N(v_{j-1})$ form an edge of $G$ with $v_i$. As a consequence, we have $$\Pb[v_iv_j\in E(G)\text{ and } E]\leq \Pb[v_iv_j\in E(G) | E]\leq 4sd^{-1/s}.$$

We now complete the proof. If $v_1, \dots, v_k$ is not an induced path in $G$, then either $E$ does not happen, or there exist some $i<j-1$ for which $v_iv_j\in E(G)$  and  $v_{j-1}\not\to v_{i}$ in $H$. Since the number of such pairs $(i, j)$ is bounded by $\binom{k}{2}$, and we have $k=\frac{1}{2s}d^{1/{2s}}$, $s\geq 2$, we conclude that
\begin{align*}
\Pb[v_1, \dots, v_k\text{ is not an induced path in }G]&\leq \Pb[\overline{E}]+\sum_{1\leq i<j-1\leq k-1}\Pb[v_iv_j\in E(G) \text{ and }E]\\
&\leq k^2 d^{1/s-1} + 4s\binom{k}{2}d^{-1/s}\leq \frac{1}{4s^2} d^{2/s-1} + \frac{1}{2s}<1. \qedhere
\end{align*}
\end{proof}

\noindent
Finally, we give the proof of the following lemma of Ne\v set\v ril and Ossona de Mendez \cite{NOdM} for completeness. Recall that a graph $G$ is \emph{$d$-degenerate} if every subgraph of $G$ contains a vertex of degree $\leq d$.

\begin{lemma}\label{lemma:degeneracy}
Let $G$ be a $d$-degenerate graph on $n$ vertices containing a Hamilton path. Then $G$ contains an induced path of length $\Omega\big(\frac{\log \log n}{\log d}\big)$. 
\end{lemma}
\begin{proof}
We start with a simple claim.
\begin{claim}\label{claim}
If $G$ is an acyclic directed graph of maximum outdegree $d$ on $n$  vertices with a directed Hamilton path, then $G$ contains an induced directed path of length at least $\Omega(\frac{\log n}{\log d})$.
\end{claim}
\begin{proof}
 Let $v_1$ be the first vertex of the Hamilton path, and  run the breadth-first search algorithm (BFS algorithm) starting at $v_1$. Since every vertex of $G$ can be reached from $v_1$ by a directed path, the algorithm explores the whole graph and constructs a spanning tree, where each node has at most $d$ children. This tree has $n$ vertices and so it must have depth at least $\Omega(\frac{\log n}{\log d})$. Moreover, by simple properties of the BFS algorithm and the acyclic property of $G$, each directed path from the root to a leaf is induced, thus giving us an induced path of the desired length.   
\end{proof}

Denote the vertices of $G$ by $v_1, \dots, v_n$ in the order they appear on the Hamilton path. Since $G$ is a $d$-degenerate graph, its edges can be directed such that the outdegree of each vertex in the resulting directed graph is at most $d$. Let $G'$ be the subgraph consisting of the edges of the Hamilton path together with all edges $v_i\to v_j$ with $i<j$, and direct all the edges of the Hamilton path from $v_i$ to $v_{i+1}$. Then $G'$ satisfies the conditions of Claim \ref{claim} (with $d+1$ instead of $d$), so it contains an induced path $P$ of length $|P|\geq \Omega(\frac{\log n}{\log d})$. 

Let the vertices of $P$ be $v_{i_1},\dots,v_{i_k}$ with $i_1<\dots<i_k$. Consider the subgraph of $G$ induced on the vertices of $P$. Since $P$ is induced in $G'$, if $v_{i_a}v_{i_b}\in E(G)$ for some $a<b-1$, then  $v_{i_b}\to v_{i_a}$ in $G$. Reverse the direction of the edges $v_{i_a}v_{i_{a+1}}$ for $a=1,\dots,k-1$, and let the resulting directed graph be $P'$. Then $P'$ satisfies the conditions of Claim \ref{claim}, so we find a path $P''\subseteq P'$ which is induced in $P'$, and thus also in $G$, of length $\Omega(\frac{\log |P|}{\log d})\geq \Omega(\frac{\log \log n}{\log d})$. This completes the proof.
\end{proof}

Now, we are ready to prove Theorem~\ref{thm:main}.

\begin{proof}[Proof of Theorem~\ref{thm:main}.]
Let $d=(\log \log n)^{2s}$. If $G$ is a $d$-degenerate graph, then by Lemma~\ref{lemma:degeneracy}, $G$ contains an induced path of length $\Omega\big(\frac{\log \log n}{\log d}\big)=\Omega\big(\frac{\log \log n}{2s\log \log \log n}\big)=\Omega\big(\frac{\log \log n}{\log\log \log n}\big)$, as claimed. If $G$ is not $d$-degenerate, it contains an induced subgraph $G'$ with minimum degree at least $d$. Since $G'$ is also $K_{s, s}$-free,  Lemma~\ref{lemma:min degree} implies that $G'$ contains an induced path of length $d^{1/2s}/2s=\Omega(\log\log n)$. \end{proof}

\end{document}